\documentclass{amsart}

\usepackage[T1]{fontenc}
\usepackage[utf8]{inputenc}

\usepackage{xcolor,hyperref}
\hypersetup{
    colorlinks,
    linkcolor={red!50!black},
    citecolor={green!40!black},
    urlcolor={blue!80!black}
  }

\usepackage{amsmath,amssymb,amsthm,mathtools,bm}
\usepackage{tikz-cd}

\usepackage[capitalize]{cleveref}

\newcommand{\from}{\colon}

\newcommand{\embed}{\hookrightarrow}

\newcommand{\overbar}[1]{\mkern
  1.5mu\overline{\mkern-1.5mu#1\mkern-1.5mu}\mkern 1.5mu}
\newcommand{\conj}[1]{\overbar{#1}}

\newcommand{\C}{\mathbb{C}}
\newcommand{\Q}{\mathbb{Q}}
\newcommand{\Z}{\mathbb{Z}}
\newcommand{\F}{\mathbb{F}}

\newcommand{\Aff}{\mathbb{A}}

\newcommand{\Crv}{\mathcal{C}}

\newcommand{\A}{\mathcal{A}}

\renewcommand{\phi}{\varphi} 
\newcommand{\eps}{\varepsilon}

\newcommand{\id}[1]{\mathfrak{#1}}

\newcommand{\p}{\mathfrak{p}}
\renewcommand{\L}{\mathcal{L}}
\renewcommand{\H}{\mathcal{H}}

\DeclareMathOperator{\GL}{GL}

\DeclareMathOperator{\Tr}{Tr}
\DeclareMathOperator{\End}{End}
\DeclareMathOperator{\NS}{NS}

\DeclareMathOperator{\im}{Im}

\DeclareMathOperator{\Rend}{\smash{\End^\dagger}}
\DeclareMathOperator{\Disc}{Disc}
\DeclareMathOperator{\Aut}{Aut}

\DeclareMathOperator{\dd}{\mathit{d}}
\DeclareMathOperator{\chr}{char}
\DeclareMathOperator{\Rec}{Rec}
\DeclareMathOperator{\Otilde}{\mathit{\smash{\widetilde{O}}}}
\newcommand{\ros}[1]{\smash{#1^\dagger}}
\renewcommand{\mod}{\ensuremath{\ \mathrm{mod}\ }}

\DeclarePairedDelimiter{\set}{\{}{\}}
\DeclarePairedDelimiter{\abs}{|}{|}
\DeclarePairedDelimiter{\pair}{\langle}{\rangle}
\DeclarePairedDelimiter{\gen}{\langle}{\rangle}

\newcommand{\mat}[4]{\left(\begin{matrix}#1&#2\\#3&#4\end{matrix}\right)}

\newtheorem{thm}{Theorem}[section]
\crefname{thm}{Theorem}{Theorems}

\crefname{lem}{Lemma}{Lemmas}
\newtheorem{prop}[thm]{Proposition}
\crefname{prop}{Proposition}{Propositions}

\theoremstyle{definition}
\newtheorem{defn}[thm]{Definition}
\crefname{defn}{Definition}{Definitions}

\newtheorem{rem}[thm]{Remark}

\crefname{algo}{Algorithm}{Algorithms}

\crefname{hyp}{Hypothesis}{Hypotheses}

\title[Counting points on abelian surfaces with Elkies's
  method]{Counting points on abelian surfaces over finite fields with
    Elkies's method}
  
\author{Jean Kieffer}

\date{\today}

\begin{document}
\maketitle

\vspace{-2em} 

\begin{abstract}
  We generalize Elkies's method, an essential ingredient in the SEA
  algorithm to count points on elliptic curves over finite fields of
  large characteristic, to the setting of p.p.~abelian surfaces. Under
  reasonable assumptions related to the distribution of Elkies primes,
  we obtain improvements over Schoof's method in two cases. If the
  abelian surface~$A$ over~$\F_q$ has RM by a fixed quadratic
  field~$F$, we reach the same asymptotic
  complexity~$\Otilde_F(\log^4 q)$ as the~SEA algorithm up to constant
  factors depending on~$F$. If~$A$ is defined over a number field, we
  count points on~$A$ modulo sufficiently many primes
  in~$\Otilde(\log^6q)$ binary operations on average. Numerical
  experiments demonstrate the practical usability of our methods.
\end{abstract}

\section{Introduction}

In this paper, we consider the problem of \emph{point counting} for
principally polarized (p.p.)~abelian varieties over finite fields:
given a p.p.~abelian variety~$A$ over~$\F_q$, we aim to compute the
characteristic polynomial of Frobenius~$\chi(A)\in \Z[X]$. If~$A$ is
an elliptic curve, this is equivalent to computing~$\# A(\F_q)$.

The motivation behind this challenge comes from different
directions. Counting points is a prerequisite for elliptic and
hyperelliptic-curve
cryptography~\cite{koblitz_EllipticCurveCryptosystems1987,
  koblitz_HyperellipticCryptosystems1989}. More recently, the hardness
of the point counting problem itself was proposed as a source of
cryptographic
protocols~\cite{dobson_TrustlessUnknownorderGroups2021}. From a
more mathematical point of view, if~$A$ is defined over a number
field, then counting points on~$A$ modulo primes of good reduction
determines the Euler factors of the $L$-function attached to~$A$.

To this date, Schoof's polynomial-time
algorithm~\cite{schoof_EllipticCurvesFinite1985,pila_FrobeniusMapsAbelian1990}
remains the central approach to point counting for abelian varieties
of dimension~$2$ or more over finite fields of large characteristic,
and much work has been devoted to making this algorithm
practical~\cite{gaudry_CountingPointsHyperelliptic2000,
  gaudry_CountingPointsGenus2011, gaudry_GenusPointCounting2012,
  abelard_ImprovedComplexityBounds2019}. In the case of abelian
surfaces over~$\F_q$, the complexity of Schoof's method
is~$\Otilde(\log^8q)$ binary operations in general,
and~$\Otilde(\log^5q)$ binary operations if the abelian surfaces
have explicit real multiplication~(RM) by a fixed quadratic
field. Note however that Schoof's approach is in competition with
cohomological algorithms, surveyed
in~\cite{kedlaya_ComputingZetaFunctions2004}, when the base field has
small characteristic; in the context of computing $L$-functions, it is
in competition with average polynomial time algorithms based on
Hasse-Witt matrices~\cite{harvey_CountingPointsHyperelliptic2014,
  harvey_ComputingHasseWittMatrices2016,
  sutherland_CountingPointsSuperelliptic2020}.

Schoof's approach is multi-modular: for a series of small
primes~$\ell \neq p$, the reduction of~$\chi(A)$ mod~$\ell$ is
computed as the characteristic polynomial of Frobenius on
the~$\ell$-torsion subgroup~$A[\ell]$, the latter being defined by
explicit polynomial equations. The algorithm stops when sufficient
information is collected to reconstruct~$\chi(A)$ using the Weil
bounds and the Chinese remainder theorem.

In the case of elliptic curves,
Elkies~\cite{elkies_EllipticModularCurves1998} showed how to
accelerate Schoof's algorithm by replacing~$A[\ell]$, the kernel of
the \emph{endomorphism}~$[\ell]$, by the kernel of
an~\emph{isogeny}~$\phi\from A\to A'$ of degree~$\ell$. Such an
isogeny will exist as soon as~$\chi(A)$ splits in linear factors
modulo~$\ell$; we say that~$\ell$ is~\emph{Elkies} in this case.
Heuristically, about half of the small primes are Elkies for a
given~$A$; this heuristic is true on average, either for all elliptic
curves over a given finite
field~\cite{shparlinski_DistributionAtkinElkies2014}, or, assuming
GRH, for all reductions of a given elliptic curve over a number field
modulo primes of good
reduction~\cite{shparlinski_DistributionAtkinElkies2015}. Then the
resulting point counting algorithm will run in~$\Otilde(\log^4q)$
binary operations, instead of~$\Otilde(\log^5q)$. Elkies's method is
an essential part of the SEA
algorithm~\cite{schoof_CountingPointsElliptic1995}, implemented in
both Pari/GP~\cite{theparigroup_PariGPVersion2019} and
Magma~\cite{bosma_MagmaAlgebraSystem1997}.

The central player in Elkies's method is the classical modular
polynomial~$\Phi_\ell$ of level~$\ell$, an explicit polynomial
equation cutting out the moduli space of pairs of~$\ell$-isogenous
elliptic curves.  More precisely, Elkies's method relies on three main
ingredients: first, upper bounds on the degree and height
of~$\Phi_\ell$~\cite{cohen_CoefficientsTransformationPolynomials1984,
  broker_ExplicitHeightBound2010}; second, an evaluation
algorithm, to compute~$\Phi_\ell(j, Y)$ as well as its derivative
$\partial_X \Phi_\ell(j, X)$ for a given value
of~$j\in\F_q$~\cite{enge_ComputingModularPolynomials2009,
  broker_ModularPolynomialsIsogeny2012,
  sutherland_EvaluationModularPolynomials2013}; and third, an isogeny
algorithm to recover~$\phi$ as an explicit rational map from this
data~\cite{elkies_EllipticModularCurves1998,
  bostan_FastAlgorithmsComputing2008}.

The recent series of papers~\cite{kieffer_DegreeHeightEstimates2022,
  kieffer_EvaluatingModularEquations2021,
  kieffer_ComputingIsogeniesModular2019} extend all these three
ingredients to the context of p.p.~abelian surfaces. In this setting,
the classical modular polynomials are replaced by modular equations
for abelian surfaces, as described
in~\cite{broker_ModularPolynomialsGenus2009,
  milio_QuasilinearTimeAlgorithm2015,
  martindale_HilbertModularPolynomials2020,
  milio_ModularPolynomialsHilbert2020}. Here we reap the benefits of
these works and describe their consequences on the point counting
problem under heuristics related to the distribution of Elkies
primes. We separate two cases depending on the moduli space of abelian
surfaces we wish to consider. In the \emph{Siegel case}, we assume
nothing a priori on our abelian surfaces; in the \emph{Hilbert case},
we fix a real quadratic field~$F$ and we only consider p.p.~abelian
surfaces with RM by its ring of integers~$\Z_F$ (but the action
of~$\Z_F$ is not assumed to be explicitly computable).  In the Hilbert
case, we reach the same asymptotic complexity as the SEA algorithm up
to constant factors depending on~$F$.

\begin{thm}
  \label{thm:main-hilbert}  
  Let~$F$ be a real quadratic field, and let~$\eps>0$. Then there
  exists an algorithm which, given a prime power~$q = p^r$
  with~$r = o(\log p)$, and given the Igusa invariants of a
  p.p.~abelian surface~$A$ over~$\F_q$ with real multiplication
  by~$\Z_F$ for which a proportion~$\eps$ of primes are Elkies
  (see \cref{def:elkies-hilbert}), computes~$\chi(A)\in \Z[X]$ in
  $\Otilde_{F,\eps}(\log^4 q)$ binary operations.
\end{thm}

In the Siegel case, it turns out that Elkies's method brings no
complexity improvement (except perhaps for logarithmic factors) over
Schoof's method for a general abelian surface~$A$
over~$\F_q$. However, it does bring an improvement when the invariants
of~$A$ admit lifts in characteristic zero of small heights. The
exponent in the complexity estimate is further decreased if we wish to
count points modulo sufficiently many primes at once.

\begin{thm}
  \label{thm:main-siegel}
  Let~$K$ be a number field, and let~$\eps>0$. Then:
  \begin{enumerate}
  \item There exists an algorithm which, given~$H\geq 0$, given a
    p.p.~abelian surface~$A$ over~$K$ whose Igusa invariants are
    well-defined and have height at most~$H$, and given a prime
    ideal~$\p$ of~$K$ of norm~$q$ such that~$A$ has good reduction
    at~$\p$ and a proportion~$\eps$ of primes are Elkies for its
    reduction~$A_\p$ (see \cref{def:elkies-siegel}),
    computes~$\chi({A_\p})\in \Z[X]$ in $\Otilde_{K,\eps}(H \log^7 q)$ binary
    operations.
  \item There exists an algorithm which, given~$H\geq 0$ and
    $q\geq 1$, given a p.p.~abelian surface~$A$ over~$K$, and
    given~$\Theta(H\log q)$ many distinct primes $\p_1,\ldots,\p_n$
    of~$K$ such that~$\log N(\p_i) = O(\log q)$, such that~$A$ has
    good reduction at all primes~$\p_i$, and such that a
    proportion~$\eps$ of primes are Elkies for each of its
    reductions~$A_{\p_i}$, computes all characteristic
    polynomials~$\chi({A_{\p_i}})\in \Z[X]$
    using~$\Otilde_{K,\eps}(\log^6 q)$ binary operations on average
    for each~$i$.
  \end{enumerate}
\end{thm}

We have released an implementation of the key step of the above
algorithms in terms of running time, namely the evaluation of modular
equations~\cite{kieffer_HDMELibraryEvaluation}, building on the C
libraries Flint~\cite{hart_FLINTFastLibrary} and
Arb~\cite{johansson_ArbEfficientArbitraryprecision2017}. This allows
us to roughly estimate the total cost of the above point-counting
algorithms in practice.

As remarked
in~\cite[§4.2.1]{dobson_TrustlessUnknownorderGroups2021},
dimension~$2$ is the largest dimension where Elkies's method can be
superior to Schoof's algorithm for generic abelian varieties, at least
in the asymptotic sense. However, Elkies's method still seems
promising in the context of counting points on p.p.~abelian varieties
with fixed RM in any dimension.

This paper is organized as follows. In \cref{sec:background}, we
quickly review previous results on Schoof's method for abelian
surfaces. In \cref{sec:siegel,sec:hilbert}, we describe Elkies's
method for p.p.~abelian surfaces in the Siegel and Hilbert case
respectively. Experimental results appear in \cref{sec:exp}. Finally,
\cref{sec:implem} presents possible directions to further reduce the
cost of point counting for abelian surfaces in practice.

\subsection*{Acknowledgements} I am deeply indebted to his former
advisors Damien Robert and Aurel Page for suggesting the thesis
project that led to this work. I also thank Noam Elkies, John Voight
and Andrew Sutherland for their insightful comments on this
work. Finally, I thank the LMFDB team for allowing me to access their
computational resources.

\section{Background on point-counting algorithms}
\label{sec:background}

\subsection{The characteristic polynomial of Frobenius}

Let~$A$ be a p.p.~abelian surface over~$\F_q$, and denote its
Frobenius endomorphism by~$\pi_A$. The
characteristic polynomial~$\chi(A)$ of~$\pi_A$ takes the form
\begin{equation}
  \label{eq:charpoly-schoof}
  \chi(A) = X^4 - s_1X^3 + (s_2+2q) X^2 - q s_1 X + q^2,
\end{equation}
where~$s_1,s_2$ are integers satisfying the following
inequalities~\cite{weil_CourbesAlgebriquesVarietes1948},
\cite[Lem.~3.1]{ruck_AbelianSurfacesJacobian1990}:
\begin{equation}
  \label{eq:weil-rueck}
  \abs{s_1}\leq 4\sqrt{q},\quad \abs{s_2}\leq 4q,
  \quad s_1^2 - 4s_2\geq 0,\quad s_2 + 4q \geq 2\abs{s_1}.
\end{equation}
Denote the Rosati involution on~$\End(A)$ induced by the principal
polarization of~$A$ by~$\dagger$. Then the \emph{real
  Frobenius}~$\psi_A = \pi_A+\ros{\pi_A}$ is an element of the
subgroup~$\Rend(A)$ of real endomorphisms of~$A$. Its characteristic
polynomial is
\begin{equation}
  \label{eq:charpoly-rm}
  \xi(A) = X^2 - s_1X + s_2.
\end{equation}

If~$\ell\neq q$ is a prime, then~$A[\ell] \subset A$ is a finite étale
group scheme isomorphic to~$(\Z/\ell\Z)^4$, and we identify it with
its set of points over an algebraic closure of~$\F_q$. The reduction
of~$\chi(A)$ modulo~$\ell$ is the characteristic polynomial of~$\pi_A$
acting on~$A[\ell]$. Recall that~$A[\ell]$ is endowed with the Weil
pairing, an alternating and nondegenerate bilinear form induced by the
principal polarization of~$A$; we denote it by
$(x,y)\mapsto \pair{x,y} \in \Z/\ell\Z$. For all~$x,y\in A[\ell]$, we
have
\begin{equation}
  \label{eq:frob-pairing}
  \pair{\pi_A(x),\pi_A(y)} = q\pair{x,y}.
\end{equation}
The Rosati involution is equal to adjunction with respect to the Weil
pairing, so~\eqref{eq:frob-pairing} translates to the equality
$\pi_A\ros{\pi_A} = q$.

Assume now that~$A$ has RM by~$\Z_F$, where~$F$ is a real quadratic
field; this means that~$A$ is equipped with an
embedding~$\Z_F\embed \Rend(A)$. Let~$\ell\in\Z$ be a prime which
splits in~$F$ in a product of two principal ideals, generated
by~$\beta,\conj{\beta}\in \Z_F$. We then have an orthogonal
decomposition
\begin{equation}
  \label{eq:rm-splitting}
  A[\ell] = A[\beta]\oplus A[\conj{\beta}],
\end{equation}
and both~$A[\beta]$ and~$A[\conj{\beta}]$ are stable
under~$\pi_A$,~$\ros{\pi_A}$ and~$\psi_A$. Since~$A[\beta]$
and~$A[\conj{\beta}]$ are not isotropic, the determinant of~$\pi_A$ on
both of these subspaces is~$q$.

\subsection{Schoof's method in dimension 2}

\label{subsec:schoof}

Recall that any p.p.~abelian surface~$A$ over~$k = \F_q$ is either a
product of two elliptic curves or the Jacobian of a hyperelliptic
genus~$2$ curve~$\Crv$ defined over~$k$. In the point-counting
context, we only have to consider this second case.
The fundamental building block of Schoof's
method~\cite{gaudry_CountingPointsHyperelliptic2000,
  gaudry_GenusPointCounting2012} is to be able to work with the
torsion subgroups~$A[\ell]$ in a computationally efficient way. Since
Elkies's method ultimately involves computations with subgroups of
Jacobians as well, the computational techniques developed for Schoof's
method will still apply in our context.

The first step is to choose birational coordinates on~$A$. A popular
choice is to consider Mumford coordinates. Let~$y^2 = P(x)$ be an
equation of~$\Crv$, and let~$K_\Crv$ be the canonical divisor
of~$\Crv$. A generic point of~$A$ is linearly equivalent to~$D-K_\Crv$
for a unique degree-two divisor~$D$ on~$\Crv$; in turn, a generic
divisor of degree two can be written as the zero locus of polynomials
of the form~$x^2 + u_1 x + u_0$ and~$y - v_1x - v_0$ in a unique
way. This defines the Mumford coordinates~$(u_0,u_1,v_0,v_1)$ as a
rational map from~$A$ to the affine space~$\Aff_k^4$, and~$A$ is
birational to its image. Denote the coordinate ring of~$\Aff_k^4$
by~$k[U_0,U_1,V_0,V_1]$.

If~$S$ is any finite subgroup of~$A$, then~$S$ is stable
by~$D\mapsto i(D)$ where~$i(D)$ denotes the hyperelliptic involution,
and hence by change of sign of~$v$-coordinates. Assume that~$S$ is
\emph{generic} in the sense that Mumford coordinates are well-defined
at all points of~$S$ and that all pairs~$\set{D,i(D)}$ in~$S$ have
distinct~$u_1$-coordinates. Then the Gröbner basis cutting out~$S$ in
terms of Mumford coordinates, in the monomial
ordering~$U_1 < U_0 < V_1 < V_0$, will take the convenient form
\begin{equation}
  \label{eq:groebner-basis}
  \begin{cases}
    V_0 - V_1 S_0(U_1) = 0,\\
    V_1^2 - S_1(U_1) = 0, \\
    U_0 - R_0(U_1) = 0, \\
    R_1(U_1) = 0
  \end{cases}
\end{equation}
for some univariate polynomials~$R_1,R_0,S_1,S_0\in k[U_1]$. The
methods of~\cite[§3]{gaudry_GenusPointCounting2012} describe how to
compute this Gröbner basis when $S = A[\ell]$, assuming that the
equation of~$\Crv$ is given by a polynomial~$P$ of degree five, in
other words that~$\Crv$ admits a rational Weierstrass
point~$\infty$. The input of their algorithm is given by Cantor's
division polynomials, which provide an explicit description of the
following composition as a rational map:
\begin{displaymath}
  \begin{tikzcd}
    \Crv \ar[rr, "{p \mapsto [p]-[\infty]}"] && A \ar[r, "{[\ell]}"] &
    A \ar[rr, "{(u_0,u_1,v_0,v_1)}"] && \Aff^4.
  \end{tikzcd}
\end{displaymath}
More generally, if~$A'$ is the Jacobian of another hyperelliptic
genus~$2$ curve~$\Crv'$ over~$k$, is~$f\from A\to A'$ is any isogeny,
and if~$p_0\in \Crv(k)$ is any rational point, then the same methods
will compute a Gröbner basis describing~$\ker(f)\subset A$ given the
explicit expression of the composed map
\begin{equation}
  \label{eq:rational-map}
  \begin{tikzcd}    
    \Crv \ar[rr, "{p \mapsto [p]-[p_0]}"] && A \ar[r, "{f}"] &
    A' \ar[rr, "{(u_0,u_1,v_0,v_1)}"] && \Aff^4.
  \end{tikzcd}
\end{equation}

If the degrees of these rational fractions is bounded above by~$d$,
then the whole Gröbner basis computation takes~$\Otilde(d^3)$
operations in~$k$, hence~$\Otilde(d^3\log q)$ binary operations. Note
that a complexity~$\Otilde(d^{3-1/\omega} \log q)$, where~$\omega$
denotes the exponent of matrix multiplication, could probably be
achieved by computing bivariate resultants using an algorithm of
Villard~\cite{villard_ComputingResultantGeneric2018a} instead of the
more classical evaluation-interpolation method. The resulting
polynomials in~\eqref{eq:groebner-basis} have degree~$O(d^2)$.

Once the Gröbner basis~\eqref{eq:groebner-basis} is known, computing
the Frobenius endomorphism on~$S$ is simply a matter of
computing~$(U_0^q,U_1^q,V_0^q,V_1^q)$ using a square-and-multiply
algorithm, reducing the result modulo the defining ideal of~$S$ at
each step, for a total cost of~$\Otilde(d^2\log^2 q)$ binary
operations.

When running Schoof's method in the generic case, one takes $S = A[\ell]$ and
$d = O(\ell^2)$. It only remains to find the correct values of~$s_1$
and~$s_2$ in~$\Z/\ell\Z$ such that Frobenius characteristic
equation~\eqref{eq:charpoly-schoof} holds on~$A[\ell]$
\cite[Alg.~1]{gaudry_GenusPointCounting2012}. The dominant step in the
whole method is the Gröbner basis computation, which accounts for the
final complexity of~$\Otilde(\log^8q)$ binary operations. In the RM
case, one can take~$S = A[\beta]$ instead provided that~$\ell$ splits
correctly in~$\Z_F$~\cite{gaudry_CountingPointsGenus2011}. Then one
has~$d = O(\ell)$, giving a total point-counting complexity
of~$\Otilde(\log^5 q)$ binary operations; both the Gröbner step and
the Frobenius computation are asymptotically
dominant. In~\cite{gaudry_CountingPointsGenus2011}, the real
multiplication action of~$\Z_F$ is assumed to be explicitly
computable; in this case the Chinese remainder theorem can be used to
recover~$\psi_A\in \Z_F$ directly.

\subsection{Non-generic cases}

The above genericity assumption on~$S$ does not necessarily hold in
general. As detailed
in~\cite[§5]{abelard_ImprovedComplexityBounds2019}, it could fail in a
finite number of different ways: certain elements of~$S$ might be of
the form~$[p]-[\infty]$ or $2[p]-K$ for some~$p\in \Crv(k)$, so that
their Mumford coordinates are not defined; or they might take the
generic form~$[p_1] + [p_2] - K$, but the Mumford coordinates
of~$\ell([p_1]-[\infty])$ might not be defined. Each of these possible
degeneracy types can be managed by writing another polynomial system
with a smaller number of variables or lower degrees than the original
one; therefore, considering the generic case is sufficient from a
complexity-theoretic point of view. It is also sufficient from a
practical point of view: in large characteristics,~$S$ will be generic
with overwhelming probability, and any particular~$\ell$ causing
problems can simply be skipped.

\section{Elkies's method for abelian surfaces: the Siegel case}
\label{sec:siegel}

\subsection{Polarized isogenies between abelian surfaces}

Modular polynomials describing~$\ell$-isogenies between elliptic
curves play a central role in Elkies's method in
dimension~$1$. Similarly, explicit equations for moduli spaces of
suitably isogenous abelian surfaces will play a central role in
Elkies's method for p.p.~abelian surfaces. Such moduli spaces only
exist for certain isogeny types that are directly related to the
polarizations and endomorphism rings of the abelian varieties we
consider.

Recall that~$\NS(A)$ denotes the Néron--Severi group of~$A$,
consisting of line bundles on~$A$ defined over an algebraic closure of
the base field, up to algebraic equivalence.

\begin{thm}[{\cite[Prop.~17.2]{milne_AbelianVarieties1986}}]
  \label{thm:End-NS}
  Let~$A$ be a principally polarized abelian variety over a
  field~$k$. Then there is a natural isomorphism of abelian groups
  \begin{displaymath}
    \L_A \from (\Rend(A), +) \to (\NS(A), \otimes).
  \end{displaymath}
  The line bundle~$\L_A(1)$ is the unique line bundle (up to algebraic
  equivalence) defining the principal polarization of~$A$; moreover,
  for each~$\beta\in \Rend(A)$, the line bundle~$\L_A(\beta)$ is ample
  if and only if~$\beta$ is totally positive.
\end{thm}

If~$A$ and~$A'$ are p.p.~abelian varieties of the same dimension
over~$k$, and if $\beta\in \Rend(A)$, we say that $f\from A\to A'$ is
a~\emph{$\beta$-isogeny} if $f^*\L_{A'}(1) = \L_A(\beta)$. In
particular,~$\beta$ must be totally positive. If~$\beta$ is moreover
prime to~$\chr(k)$, an isogeny~$f$ is a~$\beta$-isogeny if and only if
the following conditions are
satisfied~\cite[Thm.~1.1]{dudeanu_CyclicIsogeniesAbelian2017}:
\begin{enumerate}
\item \label{it:isotropic} $\ker(f)\subset A[\beta]$ and~$\ker(f)$ is
  maximal isotropic for the Weil pairing attached to this subgroup;
\item The image~$A'$ is endowed with the natural principal
  polarization of~$A/\ker(f)$ coming from the
  conditions~\eqref{it:isotropic}.
\end{enumerate}
Since~$\Rend(A)$ always contains a copy of~$\Z$, it makes sense to
talk about~$\ell$-isogenies for any~$\ell\in \Z_{\geq 1}$. In the case
of elliptic curves, this corresponds to the usual notion of cyclic
isogenies of degree~$\ell$; but in dimension~$g$ the degree of
an~$\ell$-isogeny is~$\ell^g$. In general, the degree of
a~$\beta$-isogeny is $(\deg\beta)^{1/2}$.

We say that two isogenies~$f_1\from A\to A'$ and~$f_2\from A\to A''$
are~\emph{equivalent} if there exists an isomorphism of p.p.~abelian
varieties~$\eta\from A'\to A''$ such that~$f_2 = \eta\circ f_1$.

\subsection{Elkies primes}
The goal of Elkies's method in the Siegel case is, given a
p.p.~abelian surface~$A$ over~$\F_q$ and a prime~$\ell$, to obtain
information on~$\chi(A)\mod\ell$ using an~$\ell$-isogeny~$f$, with
domain~$A$, defined over~$\F_p$. We say that~$\ell$ is~\emph{Elkies}
for~$A$ if such an~$f$ exists. Before explaining how to obtain such
an~$f$, let us describe how~$\chi(A)$ can be computed from the action
of~$\pi_A$ on~$\ker(f)$; and conversely, how the splitting behavior
of~$\chi(A)$ modulo~$\ell$ can guarantee that~$\ell$ is Elkies.

If~$P$ is a monic polynomial of degree~$d$ whose constant
coefficient~$a_0$ is invertible, we denote by~$\Rec_q(P)$ the monic
polynomial~$a_0^{-1} X^d P(q/X)$.

\begin{prop}
  \label{prop:siegel-elkies-charpoly}
  Assume that~$\ell$ is Elkies for~$A$, and let~$f\from A\to A'$ be
  an~$\ell$-isogeny defined over~$\F_q$. Let~$P\in \Z/\ell\Z[X]$ be
  the characteristic polynomial of~$\pi_A$ on~$\ker(f)$.
  Then~$\chi(A) = P \Rec_q(P)$ modulo~$\ell$.
\end{prop}

\begin{proof}
  Choose a symplectic basis of~$A[\ell]$ whose first two vectors
  generate~$\ker(f)$. By~\eqref{eq:frob-pairing}, the matrix
  of~$\pi_A$ in this basis is of the form
  \begin{displaymath}
    \mat{M}{*}{0}{qM^{-t}}
  \end{displaymath}
  for some~$M\in \GL_2(\Z/\ell\Z)$; here~$M^{-t}$ denotes the inverse
  transpose of~$M$. The characteristic polynomial of~$q M^{-t}$
  is~$\Rec_q(P)$.
\end{proof}

\begin{prop}
  \label{prop:siegel-elkies-sufficient}
  Let~$\ell$ be a prime, and assume that one of the following holds:
  \begin{enumerate}
  \item \label{it:coprime-split} $\chi(A)$ splits modulo~$\ell$ as a
    product of the form~$P\Rec_q(P)$ where the polynomials~$P$
    and~$\Rec_q(P)$ are coprime;
  \item \label{it:total-split} $\chi(A)$ is totally split modulo~$\ell$.
  \end{enumerate}
  Then~$\ell$ is Elkies for~$A$.
\end{prop}

Recall that the roots of~$\chi(A)$ over~$\C$ take the
form~$\lambda_1,q/\lambda_1,\lambda_2,q/\lambda_2$
where~$\lambda_1,\lambda_2$ are complex numbers of modulus~$\sqrt{q}$;
hence assumption~\eqref{it:coprime-split} means that~$\chi(A)$ splits
modulo~$\ell$ in two coprime degree~$2$ factors whose roots
are~$\set{\lambda_1,\lambda_2}$ and~$\set{q/\lambda_1,q/\lambda_2}$
respectively (up to a possible renaming
of~$q/\lambda_2\mapsto \lambda_2$). Merely assuming that~$\chi(A)$
splits in degree~$2$ factors is not sufficient to ensure that~$\ell$
is Elkies: for instance,~$A$ might be product of two elliptic curves
over~$\F_q$ for which~$\ell$ is an Atkin prime.

\begin{proof}  
  In case~\eqref{it:coprime-split}, define~$a = P(\pi_A)$
  and~$b = \Rec_q(P)(\pi_A)$ as endomorphisms of~$A[\ell]$. We have a
  decomposition of~$A[\ell]$ as~$\ker(a)\oplus \ker(b)$, and both
  subspaces have dimension~2. Let us show that~$\ker(a)$ is
  isotropic: by~\eqref{eq:frob-pairing},~$b$ is the adjoint of~$a$,
  hence
  \begin{displaymath}
    \pair{\ker(a), \ker(a)} = \pair{ \im(b), \im(b)}
    = \pair{A[\ell],\im (a b)} = 0.
  \end{displaymath}
  In case~\ref{it:total-split}, if~$v\in A[\ell]$ is an eigenvector
  of~$\pi_A$, then~$v^\perp \subset A[\ell]$ is still~$\pi_A$-stable.
  Therefore, there exists~$w\in v^\perp$ such
  that~$\gen{v}\oplus \gen{w}\subset A[\ell]$ is a $\pi_A$-stable
  subspace of dimension~$2$; it is isotropic by construction.
\end{proof}

Given \cref{prop:siegel-elkies-sufficient}, it is reasonable to expect
that a positive proportion of small primes~$\ell$ (in general, about
three in eight) will be Elkies for a given~$A$. This heuristic will
motivate the more precise \cref{def:elkies-siegel} later on.

\subsection{Modular equations of Siegel type}
\label{subsec:siegel-modeq}

Denote by~$\A_g$ the Siegel moduli space of p.p.~abelian varieties of
dimension~$g$, considered as an algebraic variety over~$\Q$. Then for
each prime~$\ell$, we have the following diagram of~$\Q$-varieties:
\begin{equation}
  \label{diag:siegel}
  \begin{tikzcd}
    & \A_g^0(\ell) \ar[ld, "p_1", swap] \ar[rd, "p_2"] & \\
    \A_g & & \A_g
  \end{tikzcd}
\end{equation}
where~$\A_g^0(\ell)$ denotes the coarse moduli space of pairs~$(A,K)$
where~$A$ is a p.p.~abelian variety of dimension~$g$
and~$K\subset A[\ell]$ is the kernel of an~$\ell$-isogeny. The
morphisms~$p_1$ and~$p_2$ are~$(A,K)\mapsto A$ and~$(A,K)\mapsto A/K$
respectively. Both~$p_1$ and~$p_2$ and are finite coverings;
moreover~$(p_1,p_2)$ realizes a birational isomorphism
between~$\A_g^0(\ell)$ and its image in~$\A_g\times\A_g$.

If~$g=2$, the graded~$\Q$-algebra of Siegel modular forms is free over
four
generators~$I_4,I_6',I_{10},I_{12}$~\cite{igusa_ArithmeticVarietyModuli1960}.
Therefore~$\A_2$ is a rational variety. The zero locus of~$I_{10}$
exactly corresponds to the locus of products of elliptic curves;
moreover, it is computationally convenient to work with coordinates
on~$\A_2$ which share a common, small-degree denominator, so a common
choice of coordinates~$j = (j_1,j_2,j_3)$ on~$\A_g$ is given by the
\emph{Igusa invariants}~\cite{streng_ComputingIgusaClass2014}, which
are scalar multiples of
\begin{equation}
  \label{eq:igusa}
  \frac{I_4 I_6'}{I_{10}}, \quad \frac{I_4 I_{12}}{I_{10}^2}, \quad \text{and}\quad
  \frac{I_4^5}{I_{10}^2}.
\end{equation}
The Igusa invariants define a local isomorphism from~$\A_2$
to~$\Aff^3$ at every point where both~$I_4$ and~$I_{10}$ are
nonzero. Hitting this singular locus may cause problems during the
point-counting algorithm. In this case, one can always consider
another set of coordinates on~$\A_2$; we postpone this
discussion to~§\ref{subsec:siegel-degenerate} below.

The \emph{Siegel modular equations} of level~$\ell$ are explicit
equations for the image of~$\A_2^0(\ell)$ in~$\A_2\times \A_2$. They
take the form of three multivariate rational
fractions~$\Psi_{\ell,k}\in \Q(J_1,J_2,J_3)[Y]$ for~$1\leq k\leq
3$. Writing~$j\circ p_1 = (j_1,j_2,j_3)$
and~$j\circ p_2 = (j_1',j_2',j_3')$, the equations of~$\A_2^0(\ell)$
are the following:
\begin{equation}
  \label{eq:siegel-modeq}
  \begin{cases}
    \Psi_{\ell,1}(j_1,j_2,j_3,j_1') = 0,\\
    \partial_X \Psi_{\ell,1}(j_1,j_2,j_3,j_1')\cdot j_k' -
    \Psi_{\ell,k}(j_1,j_2,j_3,j_1') = 0
    \quad \text{for $k = 2,3$}.
  \end{cases}
\end{equation}

Let~$\dd(\ell) = \ell^3+\ell^2+\ell+1$ be the degree
of~$p_1\from\A_2^0(\ell)\to\A_2$.  Then, for any p.p.~abelian
surface~$A$ over~$\C$, there are exactly~$d(\ell)$
non-equivalent~$\ell$-isogenies with domain~$A$. Let~$A'_i$
for~$1\leq i\leq \dd(\ell)$ be their codomains. Assume that
neither~$A$ nor any of the isogenous surfaces~$A_i'$ lie in the
singular locus of~$j$ as defined above; assume moreover that all
coordinates~$j_1(A_i')$ are distinct. Then the denominator of Siegel
modular equations does not vanish at~$j(A)$; moreover the roots of
Siegel modular equations evaluated at~$A$, by which we mean all
tuples~$(j_1',j_2',j_3')\in \C^3$ such that the
equations~\eqref{eq:siegel-modeq} are satisfied with~$j_k = j_k(A)$
for~$1\leq k\leq 3$, are precisely the~$d(\ell)$ tuples of the
form~$j(A_i')$ for~$1\leq i\leq d(\ell)$. The same properties holds if
we replace~$\C$ by an algebraic closure of the finite field~$\F_p$ for
any prime~$p\neq \ell$ such that~$p>3$; this can be deduced either
from the classical lifting theorems, or from the fact that both~$\A_2$
and~$\A_2^0(\ell)$ are actually smooth stacks
over~$\Z[1/\ell]$~\cite[Chap.~I,~§4]{faltings_DegenerationAbelianVarieties1990}.

\subsection{Algorithms for Siegel modular equations}

We now present the prerequisites of Elkies's method in the Siegel
case, namely an upper bound on the size of modular equations of Siegel
type, as well as algorithms that allow us to evaluate modular
equations at a given point and compute the associated
$\ell$-isogenies.

In the following result, the \emph{total degree} of a
multivariate rational fraction~$F$ over~$\Q$ means the maximum between
the total degrees of its numerator and denomiator; moreover, the
\emph{height} of~$F$ is the maximum of~$\log\abs{c}$, where~$c\in \Z$
runs through the coefficients of~$F$ written in irreducible form. The
notion of height generalizes to arbitrary number fields.

\begin{thm}[{\cite[Thm.~1.1 and
    Prop.~4.11]{kieffer_DegreeHeightEstimates2022}}]
  \label{thm:height-siegel} The degree of~$\Psi_{\ell,k}$ in~$X$
  equals~$\dd(\ell)$ for~$k=1$, and equals~$\dd(\ell)-1$
  for~$k = 2,3$. The total degree of~$\Psi_{\ell,k}$ in~$J_1,J_2,J_3$
  is bounded above by~$5\dd(\ell)/3$ for~$k=1$, and~$10\dd(\ell)/3$
  if~$k=2,3$. The height of~$\Psi_{\ell,k}$ is~$O(\ell^3\log \ell)$
  as~$\ell$ tends to infinity.
\end{thm}

\begin{thm}[{\cite{kieffer_EvaluatingModularEquations2021}}]
  \label{thm:eval-siegel}
  Let~$K$ be a fixed number field. There exists an algorithm which,
  given a prime~$\ell$, given~$H\geq 0$, and
  given~$(j_1,j_2,j_3)\in K^3$ where the denominator of Siegel modular
  equations of level~$\ell$ does not vanish, computes the polynomials
  \begin{displaymath}
    \Psi_{\ell,k}(j_1,j_2,j_3,X) \quad\text{and}\quad
    \partial_{J_i} \Psi_{\ell,k}(j_1,j_2,j_3,X)
  \end{displaymath}
  for all indices~$1\leq i,k\leq 3$ as elements of~$K[X]$ in quasi-linear
  time, in other words in~$\Otilde_K(\ell^6 H)$ binary operations.
\end{thm}

\begin{thm}[{\cite[Thm.~1.1]{kieffer_ComputingIsogeniesModular2019}}]
  \label{thm:isog-siegel}
  Let~$k$ be a field. Then there exists an algorithm which, given:
  \begin{itemize}
  \item a prime~$\ell$ such that~$\chr(k) > 8\ell +7$ if it is finite;
  \item the Igusa invariants of two~$\ell$-isogenous p.p.~abelian
    varieties~$A$ and~$A'$ defined over~$k$, lying outside the
    singular locus of~$j$, such that~$\Aut(A)$ and~$\Aut(A')$ are both
    equal to~$\set{\pm 1}$, and such that the subvariety
    of~$\Aff^3\times\Aff^3$ cut out by the
    equations~\eqref{eq:siegel-modeq} is normal at~$(j(A),j(A'))$;
  \item the nine values
    \begin{displaymath}
      \partial_{J_i}\Psi_{\ell,k}(j_1(A),j_2(A),j_3(A),j_1(A'))\in k
    \end{displaymath}
    for~$1\leq i, k\leq 3$;
  \end{itemize}
  computes the following data:
  \begin{itemize}
  \item a tower~$k'/k$ of at most three quadratic extensions;
  \item equations for two genus~$2$ hyperelliptic curves~$\Crv$
    and~$\Crv'$ over~$k'$ whose Jacobians are isomorphic to $A$
    and~$A'$ respectively;
  \item a point~$p_0\in \Crv(k')$; and
  \item four rational fractions in~$k'(x,y)$ of total degree~$O(\ell)$
    describing an~$\ell$-isogeny $f\from A\to A'$ in the sense
    of~\eqref{eq:rational-map} using~$p_0\in \Crv(k')$ as a base
    point;
  \end{itemize}
  for the cost of~$\Otilde(\ell)$ elementary operations and~$O(1)$
  square roots in~$k'$.
\end{thm}

\subsection{Complexity bounds for Elkies's method}
\label{subsec:siegel-complexity}

Let~$K$ be a fixed number field, and let~$A_0$ be a p.p.~abelian
surface over~$K$. Let~$\p$ be a prime of~$K$, of residue field~$\F_q$,
where~$A_0$ has good reduction; we denote the reduced abelian variety
over~$\F_q$ by~$A$.  In order to count points on~$A$, we apply
Elkies's method in the following way.
\begin{enumerate}
\item For a series of small primes~$\ell$, we apply
  \cref{thm:eval-siegel} to evaluate the Siegel modular equations
  at~$j(A_0)$, and reduce the result to~$\F_q$. This step
  costs~$\Otilde_K(\ell^6H)$ binary operations, and will dominate the
  rest of the algorithm; however, it needs to be done only once if we
  with to count points on~$A$ modulo~$\p_i$ for several primes~$\p_i$.
\item We then attempt to find a root of the reduced Siegel modular
  equations~\eqref{eq:siegel-modeq} over~$\F_q$; this
  costs~$\Otilde(\ell^3\log^2 q)$ binary operations. If there are
  none, we simply skip~$\ell$.
\item If we find one, then we hope that it corresponds to the Igusa
  invariants of a p.p.~abelian surface~$A'$ over~$\F_q$ for which the
  genericity conditions of \cref{thm:isog-siegel} hold. If they do,
  then~$\ell$ is Elkies for~$A$, and we are able to compute an
  explicit rational representation of such an~$f$
  using~$\Otilde(\ell\log q)$ binary operations.
\item At this point, the methods described in~§\ref{subsec:schoof}
  allow us to compute with formal points of~$\ker(f)$, and hence
  recover the characteristic polynomial of Frobenius on this subgroup;
  the Gröbner basis and Frobenius computations
  cost~$\Otilde(\ell^3\log q)$ and~$\Otilde(\ell^2\log^2 q)$ binary
  operations respectively. The polynomial~$\chi(A)\mod\ell$ itself is
  finally computed using \Cref{prop:siegel-elkies-charpoly}.
\end{enumerate}
By the Hasse-Weil bounds, carrying out this algorithm successfully for
a series of primes~$\ell_i$ such that~$\prod\ell_i > 8q$ is sufficient
to recover~$\chi(A)\in \Z[X]$ using the Chinese remainder theorem.

\subsection{Degenerate cases}
\label{subsec:siegel-degenerate}

We now analyze the different failure cases of the algorithm sketched
above. The following issues may arise for any Elkies prime~$\ell$:
\begin{enumerate}
\item \label{pb:singular-locus} One or more of the p.p.~abelian
  surfaces~$\ell$-isogenous to~$A$ over an algebraic closure of~$\F_q$
  may lie on the singular locus of~$j$.
\item \label{pb:same-j1} Several of these abelian surfaces may have
  the same~$j_1$-coordinate.
\item \label{pb:elliptic} Either~$A$ or~$A'$ may be the product of two
  elliptic curves.
\item \label{pb:automorphisms} Either~$A$ or~$A'$ may have extra automorphisms.
\item \label{pb:not-normal} The subvariety of~$\Aff^3\times\Aff^3$ cut
  out by the Siegel modular equations of level~$\ell$ may not be
  normal at~$(j(A),j(A'))$.
\end{enumerate}

Both failure cases~\eqref{pb:singular-locus} and~\eqref{pb:same-j1}
can be detected during the execution of the algorithm of
\cref{thm:eval-siegel}; they are easily solved by taking different
birational coordinates on~$\A_2$. For instance, we can apply a
projective linear transformation~$m$ with integer coefficients on the
weight~$20$
coordinates~$(I_4^5 : I_4^2 I_6'^2 : I_4^2 I_{12} : I_4 I_6' I_{10} :
I_{10}^2)$ before taking the quotients~\eqref{eq:igusa}. When
choosing~$m$, we must make sure that~$O(\ell^6)$ non-equalities in the
algebraic closure of~$\F_q$ are satisfied. This can always be achieved
provided that we choose coefficients in~$m$ of height~$O(\log
\ell)$. The degree and height bounds of \cref{thm:height-siegel} still
hold for the modified modular equations using this new set of
invariants.

In case~\eqref{pb:elliptic}, we can apply the SEA algorithm on both
factors.

In case~\eqref{pb:automorphisms}, we obtain a lot of new information
about~$A$: either~$A$ is a twist of the Jacobian of the hyperelliptic
curve~$y^2 = x^5-1$ with complex multiplication by~$\Q(\zeta_5)$, so
that~$\chi(A)$ can be determined by the CM
method~\cite{weng_ConstructingHyperellipticCurves2003}; or we can find
an explicit isogeny from~$A$ to the product of two elliptic
curves~\cite[§8]{igusa_ArithmeticVarietyModuli1960}.

Finally, in case~\eqref{pb:not-normal}, a geometric argument shows
that~$(A,A')$ must be the reduction to~$\F_q$ of a singular point in
characteristic
zero~\cite[Rem.~4.12]{kieffer_ComputingIsogeniesModular2019}. Using
the complex-analytic uniformization of~$\A_2$ then shows that such
singular points either admit extraneous automorphisms, or have the
property that there exist two non-equivalent~$\ell$-isogenies from~$A$
to~$A'$. This implies that~$A$ possesses a non-integral endomorphism
of norm~$\ell^4$. Unlike the elliptic curve case, where we would
switch to the CM method straighaway, this new piece of information
seems insufficient to describe~$\End(A)$ in a precise way in higher
dimensions. Instead, we simply skip~$\ell$ and carry on with Elkies's
method for other primes; we make the heuristic assumption that
sufficiently many Elkies primes still exist.

\begin{defn}
  \label{def:elkies-siegel}
  Let~$\eps>0$, and let~$A$ be a p.p.~abelian surface over~$\F_q$. We
  say that~$A$ has \emph{a proportion~$\eps$ of primes are Elkies
    for~$A$} if the following holds: for
  every~$X \geq \frac{1}{\eps}\log q$, the proportion of
  primes~$\ell\leq X$ such that~$\ell$ is Elkies for~$A$ and~$\End(A)$
  admits no non-integral endomorphism of norm~$\ell^4$ is at
  least~$\eps$.
\end{defn}

If a proportion~$\eps$ of primes are Elkies for~$A$, then we are able
to collect sufficiently many Elkies primes~$\ell_i$ such
that~$\ell_i\in O_\eps(\log q)$. Thus, \Cref{thm:main-siegel}
directly follows from the complexity estimates given
in~§\ref{subsec:siegel-complexity}.

\begin{rem}
  It is known that for any fixed~$\eps>0$, there exists an elliptic
  curve~$E$ over some finite field~$\F_q$ for which it does not hold
  that a proportion~$\eps$ of primes are Elkies
  for~$E$~\cite{shparlinski_ProductSmallElkies2015}.
  Following~\cite{sutherland_EvaluationModularPolynomials2013}, we
  could relax \cref{def:elkies-siegel} by taking an upper bounds of
  the form $X = \frac1\eps\log q\log\log(q)^n$ for some
  fixed~$n\geq 1$. Then we can hope that there exists a
  positive~$\eps>0$ such that all abelian surfaces over finite fields
  have a proportion~$\eps$ of Elkies primes. This more permissive
  definition does not modify the complexity estimates of
  \cref{thm:main-siegel}.
\end{rem}

\section{Elkies's method for abelian surfaces: the Hilbert case}
\label{sec:hilbert}

\subsection{Elkies primes}

In this section, we fix a real quadratic field~$F$ of
discriminant~$\Delta_F$, and we consider the point counting problem
for a p.p.~abelian surface~$A$ over~$\F_q$ with real multiplication
by~$\Z_F$. By \cref{thm:End-NS}, this real multiplication structure
guarantees the presence of supplementary isogenies compared to the
Siegel case: for each totally positive~$\beta\in \Z_F$, we can look
for~$\beta$-isogenies $f\from A\to A'$ defined over~$\F_q$. Assume
further that~$\beta\in \Z_F$ is prime, prime to~$\Delta_F$, and
that~$\ell = N_{F/\Q}(\beta)\in \Z$ is also prime; in other
words~$\ell$ is a prime that splits in~$F$ in a product of two
ideals~$(\beta)\cdot(\conj{\beta})$ that are trivial in the narrow
class group of~$\Z_F$. By the \v{C}ebotarev density theorem, this kind of
splitting will occur for a positive proportion of
primes~$\ell\in\Z$. We say that~$\beta$ is \emph{Elkies} for~$A$ if
a~$\beta$-isogeny~$f$ with domain~$A$ exists
over~$\F_q$. Then~$\ker(f)\subset A[\ell]$ is a~$\pi_A$-stable
subgroup of order~$\ell$; therefore we can hope to obtain information
on~$\chi(A)$ mod~$\ell$ by manipulating polynomials of
degree~$O(\ell)$ only, as in Elkies's original method for elliptic curves.

In the Hilbert case, the Chinese remaindering step is formulated in
terms of the real Frobenius endomorphism~$\psi_A = \pi_A+\ros{\pi_A}$
as an element of~$\Z_F$.  By the Weil bounds~\eqref{eq:weil-rueck}, we
have
\begin{equation}
  \label{eq:weilbounds-rm}
  |\Tr(\psi_A)|\leq 4\sqrt{q} \quad \text{and}\quad \Disc(\Z[\psi_A])\leq 4q.
\end{equation}
Assume that~$\beta$ is Elkies, and the
eigenvalue~$\lambda\in \Z/\ell\Z$ of~$\pi_A$ on~$\ker(f)$ has been
computed. Then we have
\begin{displaymath}
  \psi_A = \lambda + q/\lambda \mod\beta
\end{displaymath}
under the canonical isomorphism~$\Z_F/\beta\Z_F\simeq \Z/\ell\Z$.

In the algorithm, we consider a series totally positive Elkies
primes~$\beta_i$ in $\Z_F$, with norms~$\ell_i\in \Z$.  We collect the
values of $\psi_A$ modulo~$\beta_i$ as elements of $\Z/\ell_i\Z$ as
above.  The Chinese remainder theorem in~$\Z_F$ allows us to
reconstruct the value of $\psi_A$ modulo the ideal
$\id{B} = \prod_i (\beta_i)$.  The cost of this reconstruction is
negligible when compared to the rest of the algorithm.

\begin{prop}
  \label{prop:crt-hilbert}
  Assume that $N(\id{B})> 16q$. Then $\psi_A$ is uniquely determined by
  equations~\eqref{eq:weilbounds-rm} and the data of $\psi_A\mod{\id{B}}$.
\end{prop}

\begin{proof}
  Assume that~$\id{B}$ contains a nonzero~$\alpha\in\Z_F$ such
  that $|\Tr_{F/\Q}(\alpha)|\leq 8\sqrt{q}$ and
  $\Disc(\Z[\alpha])\leq 16q$. Then we have
  \begin{displaymath}
    N(\id{B})\leq |N_{F/\Q}(\alpha)| = \tfrac{1}{4}|\Tr(\alpha)^2 - \Disc(\Z[\alpha])| \leq 16q.
    \qedhere
  \end{displaymath}
\end{proof}

Once~$\psi_A\in \Z_F$ has been determined, its characteristic
polynomial completely describes~$\chi(A)$, as
equations~\eqref{eq:charpoly-schoof} and~\eqref{eq:charpoly-rm}
show. Heuristically, we expect that roughly half of the suitable
primes~$\beta$ will be Elkies: indeed~$\beta$ is Elkies if and only if
the characteristic polynomial of~$\pi_A$ on~$A[\beta]$, a polynomial
of degree~$2$, splits in~$\Z/\ell\Z$.

\subsection{Modular equations of Hilbert type} The key fact that
allows us to reach similar point-counting complexities in the RM case
and the case of elliptic curves is that the associated \emph{Hilbert
  modular equations} have a reasonable size.

Denote by~$\A_{2,F}$ the Hilbert moduli space of p.p.~abelian surfaces
with real multiplication by~$\Z_F$, seen as an algebraic variety
over~$\Q$. For each~$\beta$ as above, we have a diagram
of~$\Q$-varieties
\begin{equation}
  \label{diag:hilbert}
  \begin{tikzcd}
    & \A_{2,F}^0(\beta) \ar[ld, "p_1", swap] \ar[rd, "p_2"] & \\
    \A_{2,F} & & \A_{2,F}
  \end{tikzcd}
\end{equation}
where~$\A_{2,F}^0(\beta)$ denotes the coarse moduli space of
pairs~$(A,K)$ where~$A$ is a p.p.~abelian surface with real
multiplication by~$\Z_F$, and~$K\subset A[\beta]$ is the kernel of
a~$\beta$-isogeny. The Hilbert modular equations of level~$\beta$ are
explicit equations for the image of~$\A_{2,F}^0(\beta)$
in~$\A_{2,F}\times\A_{2,F}$. To define them, we make a choice of
coordinates~$j = (j_1,j_2,j_3)$ on~$\A_{2,F}$, related by an explicit
equation of the form
\begin{displaymath}
  E(j_1,j_2,j_3) = 0.
\end{displaymath}
Assume further that~$j_1$ and~$j_2$ are algebraically independent, and
write~$e = \deg_{j_3}(E)$. The Hilbert modular equations are then the
data of the three multivariate rational
fractions~$\Psi_{\beta,k}\in \Q(J_1,J_2)[J_3,X]$ of degree at
most~$e-1$ in~$J_3$ such that the system of
equations~\eqref{eq:siegel-modeq} holds with~$\ell$ replaced
by~$\beta$.

In concrete cases, it is sometimes convenient to modify this
definition and consider \emph{symmetric} modular equations on Humbert,
rather than Hilbert, surfaces. Recall that there exists a forgetful
map~$\A_{2,F}\to \A_2$ which is generically~$2$-$1$. The image~$\H_F$
of~$\A_{2,F}$, called the Humbert surface attached to~$\Z_F$, is often
less geometrically complicated than~$\A_{2,F}$. Explicit coordinates
on~$\H_F$ are also easier to describe: for instance, the Igusa
invariants~\eqref{eq:igusa} are always a valid choice. If the
discriminant of~$F$ is less than~$100$, then~$\H_F$ is rational, and
explicit parametrizations appear
in~\cite{elkies_K3SurfacesEquations2014}. In the
case~$F = \Q(\sqrt{5})$, the Gundlach invariants denoted by~$g_1,g_2$
(see \cite[Satz~6]{gundlach_BestimmungFunktionenZur1963}, although
other normalizations are also used) are convenient coordinates
on~$\H_F$ derived from an explicit description of the associated
graded ring of symmetric Hilbert modular forms.

We will denote the (symmetric) Hilbert modular equations in Igusa
invariants by~$\Psi_{\beta,k}^J$ for~$1\leq k\leq 3$; they are equal
for the prime~$\beta$ and its real
conjugate~$\conj{\beta}$. Similarly, we denote the Hilbert modular
equations of level~$\beta$ in Gundlach invariants
for~$F = \Q(\sqrt{5})$ by~$\Psi_{\beta,k}^G$ for~$k = 1,2$; they are
multivariate rational fractions in~$\Q(G_1,G_2)[Y]$.  Modular
equations on Hilbert surfaces describe~$\beta$-isogenies between
abelian surfaces with~RM by~$\Z_F$, in a similar way as
in~§\ref{subsec:siegel-modeq} for modular equations of Siegel type. In
the symmetric case, modular equations describe~$\beta$-
and~$\conj{\beta}$-isogenies simultaneously.

\subsection{Algorithms for Hilbert modular equations}

Let~$\dd(\beta) = \ell +1$ be the degree of~$p_1$ in
diagram~\eqref{diag:hilbert}. We fix a choice of
coordinates~$(j_1,j_2,j_3)$ on~$\A_{2,F}$.

\begin{thm}[{\cite[Thm.~1.1 and
    Prop.~4.13]{kieffer_DegreeHeightEstimates2022}}]
  \label{thm:height-hilbert}
  The degree of~$\Psi_{\beta,k}$ in~$X$ is~$\dd(\beta)$ for~$k=1$,
  and~$\dd(\beta-1)$ for~$k>1$. The total degrees of~$\Psi_{\beta,k}$
  are~$O_F(\ell)$, and their heights are~$O_F(\ell \log \ell)$.  In
  the case of $F = \Q(\sqrt{5})$, the total degree of~$\Psi_{\beta,k}^G$
  in~$G_1,G_2$ is at most~$10 \dd(\beta)/3$ for~$k = 1,2$.
\end{thm}

\begin{thm}[{\cite[Thm.~5.3]{kieffer_EvaluatingModularEquations2021}}]
  \label{thm:eval-hilbert}
  Let~$q = p^r$ be a prime power, and let~$F = \Q(\sqrt{5})$. There
  exists an algorithm which, given~$(g_1,g_2)\in \F_q^2$ where the
  denominator of~$\Psi_{\beta,k}^G$ for~$k = 1,2$ does not vanish,
  computes the modular equations~$\smash{\Psi_{\beta,k}^G}(g_1,g_2,X)$
  as well as their
  derivatives~$\partial_{G_i} \Psi_{\beta,k}^G(g_1,g_2,X)$
  for~$1\leq i,k\leq 2$ as elements of~$\F_q[X]$
  in~$\Otilde(\ell^2 r^2\log p)$ binary operations.
\end{thm}

This result generalizes to any other real quadratic field~$F$ for
which explicit generators of the graded rings of Hilbert modular forms
over~$\Z$ are known.  Otherwise, the evaluation algorithm can still be
run, but it involves a heuristic reconstruction of rational numbers
from their complex approximations.

\begin{thm}[{\cite[Thm.~6.3]{kieffer_ComputingIsogeniesModular2019}}]
  \label{thm:isog-hilbert}
  Let~$F$ be a fixed real quadratic field. Then there exists an open
  subvariety~$U\subset \A_{2,F}^0(\beta)$ and an algorithm which, for any
  field~$k$, given:
  \begin{itemize}
  \item a totally positive~$\beta\in \Z_F$ such
    that~$\chr(k) > 4\Tr_{F/\Q}(\beta)+7$ if it is positive;
  \item the Igusa invariants of two~$\beta$-isogenous p.p.~abelian
    surfaces~$A$ and~$A'$ defined over~$k$ with real multiplication
    by~$\Z_F$, such that this~$\beta$-isogeny comes from a point
    of~$U$;
  \item The nine values
    \begin{displaymath}
      \partial_{J_i}\Psi_{\beta,k}^J(j_1(A),j_2(A),j_3(A),j_1(A'))\in k
    \end{displaymath}
    for~$1\leq i,k\leq 3$;
  \end{itemize}
  computes the following data:
  \begin{itemize}
  \item a tower~$k'/k$ of at most three quadratic extensions;
  \item equations for two genus~$2$ hyperelliptic curves~$\Crv$
    and~$\Crv'$ whose Jacobians are isomorphic to~$A$ and~$A'$ over an
    algebraic closure of~$k$;
  \item a point~$p_0\in \Crv(k')$; and
  \item at most four possible tuples of rational fractions
    in~$k'(x,y)$ of total degree $O(\Tr_{F/\Q}(\beta))$, such that one
    of these tuples describes a~$\beta$-isogeny~$f\from A\to A'$ in the sense
    of~\eqref{eq:rational-map} using~$p_0\in \Crv(k')$ as a base
    point;
  \end{itemize}
  using~$\Otilde(\Tr_{F/\Q}(\beta)) + O_F(1)$ elementary operations
  and~$O(1)$ square roots in~$k'$.
\end{thm}

As in \cref{thm:isog-siegel}, the open
subvariety~$U\subset \A_{2,F}^0(\beta)$ in \cref{thm:isog-hilbert} can
be described explicitly. It is sufficient to impose the following
conditions~\cite[§4.2.3]{kieffer_ComputingIsogeniesModular2019}:
\begin{itemize}
\item Both~$A$ and~$A'$ have no extraneous automorphisms as
  p.p.~abelian surfaces;
\item There exists only one isogeny~$f\from A\to A'$ over an algebraic
  closure of~$k$ whose kernel is cyclic of degree~$\ell$, up to
  equivalence;
\item There exists only one possible real multiplication embedding
  of~$\Z_F$ inside both~$\End(A)$ up to real conjugation on~$\Z_F$,
  and the same holds for~$A'$;
\item Both~$A$ and~$A'$ lie outside of the singular locus of Igusa
  invariants.
\end{itemize}

\subsection{Complexity bounds} Let~$q = p^r$ be a prime power, and
let~$A$ be a p.p.~abelian surface over~$\F_q$ with RM by~$\Z_F$.  We
apply Elkies's method as follows.
\begin{enumerate}
\item Let~$\ell\in \Z$ be a prime with the correct splitting behavior
  in~$\Z_F$, and let $\beta\in \Z_F$ be a totally positive prime
  above~$\ell$.  By~\cite[Lem.~1]{gaudry_CountingPointsGenus2011}, it
  is possible to choose~$\beta$ such
  that~$\Tr_{F/\Q}(\beta)\in O_F(\sqrt{\ell})$. We evaluate the
  corresponding modular equations using \cref{thm:eval-hilbert}, for
  instance in Igusa invariants. Assuming that~$r = o(\log p)$, this
  costs~$\Otilde_F(\ell^2 \log q)$ binary operations.
\item We then attempt to find a root of these modular equations
  over~$\F_q$; this costs~$\Otilde(\ell\log^2 q)$ binary
  operations. If there are none, we  skip~$\ell$.
\item If we find one, we attempt to compute a cyclic
  isogeny~$f\from A\to A'$ of degree~$\ell$ using
  \cref{thm:isog-hilbert}. If this computation is successful, we
  obtain the rational representation of~$f$ as
  in~\eqref{eq:rational-map} in terms of rational functions of
  degree~$O_F(\sqrt{\ell})$. This costs~$\Otilde_F(\sqrt{\ell})$
  binary operations.
\item After that, computing a Gröbner basis describing~$\ker(f)$
  costs~$\Otilde(\ell^{3/2}\log q)$ binary operations; the resulting
  polynomials have degree~$O(\ell)$. Computing the Frobenius
  eigenvalue on~$\ker(f)$ costs~$\Otilde(\ell\log^2 q)$ binary
  operations.
\end{enumerate}
By \cref{prop:crt-hilbert}, the total algorithm
costs~$\Otilde_F(\log^4 q)$ binary operations provided that
sufficiently many Elkies primes exist and the computations of
\cref{thm:isog-hilbert} succeed sufficiently often.

\subsection{Degenerate cases}

The treatment of degenerate cases which may occur the algorithm above
is similar to the Siegel case. The only new possible problems are the
following:
\begin{enumerate}
\item \label{pb:rm-cm} The algorithm may involve a p.p.~abelian
  surface~$A$ that corresponds to a point where the map from~$\A_{2,F}$
  to~$\A_2$ is not étale;
\item \label{pb:rm-end} $\End(A)$ may contain an element of
  norm~$\ell^2$ outside~$\Z_F$;
\item \label{pb:parasites} The isogeny algorithm of
  \cref{thm:isog-hilbert} may output several possibilities for the
  rational representation of~$f$.
\end{enumerate}
In case~\eqref{pb:rm-cm}, we can always consider coordinates on the
Hilbert surface~$\A_{2,F}$ instead.  As in the Siegel case, we
incorporate the assumtion that case~\eqref{pb:rm-end} does not happen
too often into \cref{def:elkies-hilbert} below. Finally, in
case~\eqref{pb:parasites}, we can check which of the candidate values
actually describes a group morphism between Jacobians; if more than
one candidate passes this test, we are led back to
case~\eqref{pb:rm-end}.

\begin{defn}
  \label{def:elkies-hilbert}
  Let~$\eps>0$, and let~$A$ be a p.p.~abelian surface over~$\F_q$ with
  real multiplication by~$\Z_F$. We say that \emph{a proportion~$\eps$
    of primes are Elkies for~$A$} if the following holds: for
  every~$X\geq \frac{1}{\eps}\log q$, the proportion of
  primes~$\ell\leq X$ such that
  \begin{itemize}
  \item $\ell$ splits in the form~$(\beta)(\conj{\beta})$ for some
    totally positive~$\beta\in \Z_F$,
  \item one of~$\beta$ or~$\conj{\beta}$ is Elkies, and
  \item $\End(A)$ admits no non-real endomorphism of norm~$\ell^2$,
  \end{itemize}
  is at least~$\eps$.
\end{defn}

\begin{rem}
  Let~$\ell$ be a prime that splits in~$\Z_F$, but whose prime
  factors~$\mathfrak l\cdot \conj{\mathfrak l}$ are nontrivial in the
  narrow class group of~$\Z_F$. Instead of skipping~$\ell$ altogether
  in the point-counting algorithm, we can compute a totally positive
  generator~$\beta$ of~$\mathfrak l \cdot \mathfrak c$
  where~$\mathfrak c$ denotes a small representative of the relevant
  class in the narrow class group of~$\Z_F$. Elkies's method will also
  apply to the non-prime~$\beta$.
\end{rem}

\section{Experimental results}
\label{sec:exp}

We have implemented algorithms to evaluate modular equations for
p.p.~abelian surfaces over~$\Q$ and prime finite fields in
C~\cite{kieffer_HDMELibraryEvaluation}. The experiments presented here
can be reproduced by downloading the library and running \texttt{make
  reproduce}. In practice, we expect that the evaluation of modular
equations will exceed the cost of other polynomial manipulations in
Elkies's method by a large margin.

In the Siegel case, we consider the ``random'' tuple of Igusa
invariants of small height given by~$(159/239,-19/28,-193/246)$. We
record the time to evaluate Siegel modular equations at this point for
prime levels~$\ell \leq 17$ on a single core (AMD EPYC 7713), and
compare it with the cost estimation
of~$0.002\, \ell^6 \log(\ell)^3\log\log(\ell)$ seconds.
{\small
\begin{center}
  \begin{tabular}{c|ccccccc}
    $\ell$ & 2 & 3 & 5 & 7 & 11 & 13 & 17 \\
    Time (s) & $1.34$
               & $5.12$
                   & $96.7$
                       & $1.23\cdot 10^3$
                           & $3.97\cdot 10^4$
                                & $1.57\cdot 10^5$
                                     & $1.12\cdot 10^6$ \\
    Estimation
           & -
               & -
                   & $62$
                       & $1.2\cdot 10^3$
                           & $4.3\cdot 10^4$
                                & $1.5\cdot 10^5$
                                  & $1.1\cdot 10^6$
  \end{tabular}
\end{center}
}

In light of these results, a point-counting approach based exclusively
on Elkies's method for general p.p.~abelian surfaces would be unlikely
to beat Schoof's method in practical instances. However, using modular
equations would still allow one to introduce several improvements
inspired from the SEA algorithm (see~§\ref{sec:implem}).

In the Hilbert case for~$F = \Q(\sqrt{5})$, we consider the pair of
``random'' Gundlach invariants of small height given by
$(-117/64, -199/172)$. We evaluate Hilbert modular equations of
level~$\beta$, where~$\beta\in \Z_F$ is totally positive of prime
norm~$\ell\leq 500$ at that point, and compare it with the estimation
of~$2\cdot 10^{-4}\, \ell^2 \log(\ell)^3 \log\log\ell $ seconds.
{
  \small
  \begin{center}
    \begin{tabular}{c|cccccccc}
      $\ell$ & 11 & 19 & 29 & 31 & 41 & 59 & 61 & 71 \\
      Time (s) & $2.45$
                  & $4.14$
                       & $9.66$
                            & $11.1$
                                 & $25.6$
                                      & $59.6$
                                           & $64.0$
                                                & $107$
                                                               \\
      Estimation
             & -
                  & $2.0$
                       & $7.8$
                            & $9.6$
                                 & $23$
                                      & $66$
                                           & $73$
                                                & $113$
      \\[1pt] \hline \\[-1em]
      $\ell$ & 79 & 89 & 101 & 109 & 131 & 139 & 149 & 151  \\
      Time (s) 
             & $141$
                  & $190$
                       & $256$
                            & $315$
                                 & $562$
                                      & $673$
                                           & $794$
                                                & $709$      \\
      Estimation
             & $153$
                  & $215$
                       & $307$
                            & $379$
                                 & $630$
                                      & $741$
                                           & $896$
                                                & $929$
    \end{tabular}
  \end{center}
}
\noindent
For larger primes, counting in core-hours is perhaps more readable.
{
  \small
  \begin{center}
    \setlength{\tabcolsep}{4pt}
    \begin{tabular}{c|cccccccccc}
      $\ell$ & 179 & 181 & 191 & 199  & 211 & 229 & 239 & 241 & 251 & 269\\
      Time (h)  & $0.299$
                   & $0.307$
                         & $0.348$
                               & $0.388$
                                      & $0.452$
                                            & $0.556$
                                                  & $0.613$
                                                        & $0.623$
                                                              & $0.697$
                                                                    & $0.948$
      \\
      Estimation
             & $0.41$ & $0.42$ & $0.49$ & $0.54$ & $0.64$ & $0.79$ & $0.89$ & $0.91$ & $1.0$ & $1.2$
      \\[1pt] \hline \\[-1em]
      $\ell$  & 271 & 281 & 311 & 331 & 349 & 359 & 379 & 389 & 401 & 409  \\
      Time (h) & $0.965$
                   & $1.06$
                         & $1.37$
                               & $1.56$
                                      & $1.78$
                                            & $1.93$
                                                  & $2.37$
                                                        & $2.80$
                                                              & $3.03$
                                                                    & $3.26$
      \\
      Estimation
      & $1.2$ & $1.4$ & $1.8$ & $2.1$ & $2.4$ & $2.6$ & $3.0$ & $3.2$ & $3.4$ & $3.6$
      \\[1pt] \hline \\[-1em]
      $\ell$ & 419 & 421 & 431 & 439 & 449 & 461 & 479 & 491 & 499\\
      Time (h) & $3.47$
                   & $3.50$
                         & $3.71$
                               & $4.03$
                                      & $4.04$
                                            & $5.46$
                                                  & $4.67$
                                                        & $4.98$
                                                              & $6.14$
      \\
      Estimation
      & $3.9$ & $3.9$ & $4.2$ & $4.4$ & $4.6$ & $4.9$ & $5.5$ & $5.8$ & $6.1$
    \end{tabular}
  \end{center}
}
Consider the problem of counting points on a p.p.~abelian surface~$A$
over a prime field~$\F_q$ with~$\log q\simeq 512$ given by these
Gundlach invariants.  Assuming that half of the primes~$\beta$ are
Elkies, a strategy based purely on Elkies's method would involve all
primes~$\ell$ up to roughly~$800$ with the correct splitting behavior
in~$\Z_F$. We can roughly estimate a total cost of a few core-weeks
for this computation, thus indicating that Schoof's method can perhaps
be beaten in this context~\cite[§5.2]{gaudry_CountingPointsGenus2011}.

Finally, we report on the time (in seconds) to evaluate modular
equations for Hilbert type in Igusa invariants for~$\ell\leq 50$, at
the point given by the parameters~$(-239/152, 224/193)$ in the
different parametrizations of Humbert surfaces found
in~\cite{elkies_K3SurfacesEquations2014}, for all real quadratic
fields of discriminants up to~$100$.
{ 
  \small
  \begin{center}
    \setlength{\tabcolsep}{2.5pt}
    \begin{tabular}{c|ccccccccccccccc}
      $\Delta_F,\ell$
      & 2 & 3 & 5 & 7 & 11 & 13 & 17 & 19 & 23 & 29 & 31 & 37 & 41 & 43 & 47\\\hline
      5 &  - & - & - & - & $9.76$ & - & - & $30.2$ & - & $78.0$ & $91.1$ & - & $205$ & - & - \\
      8 & - & - & - & $2.67$ & - & - & $26.9$ & - & $48.2$ & - & $95.1$ & - & $208$ & - & $290$ \\
      12 & - & - & - & - & - & $40.9$ & - & - & - & - & - & $451$ & - & - & - \\
      13 & - & $2.02$ & - & - & - & - & $27.5$ & - & $51.1$ & $83.9$ & - & - & - & $240$ & - \\
      17 & $4.34$ & - & - & - & - & $57.2$ & - & $126$ & - & - & - & - & - & $735$ & $825$ \\
      21 & - & - & - & - & - & - & - & - & - & - & - & $363$ & - & $505$ & - \\
      24 &  - & - & - & - & - & - & - & $37.5$ & - & - & - & - & - & $235$ & - \\
      28 &  - & - & - & - & - & - & - & - & - & $211$ & - & $413$ & - & - & -\\
      29 &  - & - & $4.75$ & $5.76$ & - & $15.0$ & - & - & $48.3$ & - & - & - & - & - & -\\
      33 &  - & - & - & - & - & - & - & - & - & - & $255$ & $430$ & - & - & -\\
      37 &  - & $3.04$ & - & $13.5$ & $29.5$ & - & - & - & - & - & - & - & $503$ & - & $680$\\
      40 &  - & - & - & - & - & - & - & - & - & - &$212$ & - & $462$ & - & -\\
      41 &  $5.27$ & - & $10.5$ & - & - & - & - & - & $326$ & - & $625$ & $1050$ & - & $1450$ & -\\
      44 &  - & - & $4.77$ & - & - & - & - & - & - & - & - & $403$ & - & - & -\\
      53 &  - & - & - & $24.7$ & $49.6$ & $59.7$ & $107$ & - & - & $281$ & - & $524$ & - & $700$ & $854$ \\
      56 &  - & - & - & - & $27$ & - & - & - & - & - & - & - & - & $1360$ & -\\
      57 &  - & - & - & $30.0$ & - & - & - & - & - & - & - & - & - & $1980$ & -\\
      60 &  - & - & - & - & - & - & - & - & - & - & - & - & - & - & -\\
      61 &  - & $8.34$ & $13.9$ & - & - & $49.1$ & - & $108$ & - & - & - & - & $580$ & - & $772$ \\
      65 &  - & - & - & - & - & - & - & - & - & $1500$ & - & - & - & - & -\\
      69 &  - & - & - & - & - & $41.8$ & - & - & - & - & $253$ & - & - & - & -\\
      73 &  $6.40$ & $14.9$ & - & - & - & - & - & $241$ & $342$ & - & - & $1090$ & $1370$ & - & - \\
      76 &  - & - & $6.9$ & - & - & - & $55.8$ & - & - & - & - & - & - & - & -\\
      77 &  - & - & - & - & - & - & - & - & $133$ & - & - & $429$ & - & - & -\\
      85 &  - & - & - & - & - & - & - & $87.1$ & - & - & - & - & - & - & -\\
      88 &  - & $18.1$ & - & - & - & - & - & - & - & - & - & - & - & - & -\\
      89 &  $3.55$ & - & $15.4$ & - & $57.3$ & - & $147$ & - & - & - & - & - & - & - & $1600$\\
      92 &  - & - & - & - & - & $104$ & - & - & - & $567$ & - & - & $1370$ & - & -\\
      93 &  - & - & - & $16.5$ & - & - & - & $95$ & - & - & - & - & - & - & -\\
      97 &  $7.01$ & $8.48$ & - & - & $58.1$ & - & - & - & - & - & $1370$ & - & - & $3210$ & $4020$\\          
    \end{tabular}
  \end{center}
}

On the first line, we observe that Hilbert modular equations in Igusa
invariants for~$F = \Q(\sqrt{5})$ are indeed more expensive to
evaluate than their counterparts in Gundlach invariants.

\section{Perspectives}
\label{sec:implem}

In this final section, we sketch possible improvements to Elkies's
method for abelian surfaces as described above, following existing
works in the dimension~$1$ case. They would either reduce the constant
factors hidden in complexity estimates by large amounts, or introduce
exponential-time gains.

\subsection{Smaller modular equations}

The modular equations of Siegel and Hilbert type presented above are
higher-dimensional analogues of the classical modular
polynomials~$\Phi_\ell$ in dimension~$1$. It is well-known that other
kinds of modular polynomials provide explicit equations for
essentially the same modular curve which are much smaller, despite
sharing the same~$O(\ell^3\log\ell)$ size asymptotic: see for
instance~\cite[§3]{enge_ClassInvariantsCRT2010} and the data available
at~\cite{sutherland_DatabaseModularPolynomials}.
In the dimension~$2$ case, modular equations written in terms of theta
constants are considerably smaller than Siegel or Hilbert modular
equations as defined
above~\cite{milio_DatabaseModularPolynomials}. One can ask whether
this choice of coordinates is the optimal one.

More generally, it could well be that systems of equations of the
form~\eqref{eq:siegel-modeq} inherently force modular equations to
have large coefficients; other ways of describing the
diagram~\eqref{diag:siegel} might lead to smaller polynomials -- for
instance, a more intrinsic equation for~$\A_2^0(\ell)$ along with the
Atkin--Lehner involution exchanging~$p_1$ and~$p_2$. Such equations do
not even have to be defined by a formula valid for each~$\ell$; all we
need is an algorithm to compute such equations when~$\ell$ is given,
perhaps by computing a basis of Siegel modular forms of
level~$\Gamma^0(\ell)$, or Hilbert modular forms of
level~$\Gamma^0(\beta)$, on the fly.

\subsection{Other SEA strategies}

In the case of elliptic curves, there is more to the SEA algorithm
than applying Elkies's method to a series of distinct primes. We list
some of the possible improvements below. Due to the larger implied
constants in complexity estimates about modular equations, we expect
these improvements to have an even larger impact on practical running
times in higher dimensions.
\begin{enumerate}
\item \emph{Isogeny chains.} In favorable situations, modular
  polynomials of level~$\ell$ can be used to compute not only
  an~$\ell$-isogeny~$E\to E'$, but a chain
  of~$\ell$-isogenies~$E \to E_1\to \cdots \to E_r$ whose composition
  is an~$\ell^r$-isogeny, for
  some~$r\geq 2$~\cite{fouquet_IsogenyVolcanoesSEA2002}. This yields
  the value of~$\chi(E)$ modulo~$\ell^r$. In order to remain within
  the same complexity bound, one should take~$r$ no greater
  than~$\log\log(q)/\log(\ell) + O(1)$. The existence of an isogeny
  chain of the desired length over~$\F_q$ depends on the shape of the
  connected component of the~$\ell$-isogeny graph on which~$E$ lies.
  Note that a chain of length~$r=2$ can be constructed by evaluating
  modular polynomials only once, at~$E_1$.  In dimension~$2$, this
  strategy seems easier to apply in the Hilbert case, since isogeny
  graphs are still volcanoes in this
  case~\cite{ionica_IsogenyGraphsMaximal2020} and the composition of a
  non-backtracking chain of~$\beta$-isogenies will always be
  a~$\beta^r$-isogeny. This property does not hold
  for~$\ell$-isogenies in the Siegel case, and the shape of the
  associated isogeny graphs is also more
  complicated~\cite{brooks_IsogenyGraphsOrdinary2017}.
  
\item \emph{Atkin's method.} It is known that studying the
  factorization patterns of modular equation of level~$\ell$
  over~$\F_q$, even in the absence of rational roots, restrict the
  possible Frobenius eigenvalues
  modulo~$\ell$~\cite[§6]{schoof_CountingPointsElliptic1995},
  \cite{ballentine_IsogeniesPointCounting2016}. This allows one to
  take advantage of Atkin (i.e.~non-Elkies) primes as well. This
  information can be used at the end of the point-counting algorithm
  in an exponential-time sieve, whose practical effect is to reduce
  the number of necessary Elkies primes. In general, if we can
  compute~$n\geq 2$ possible values of~$\chi(A)$ modulo~$\ell$, the
  ``value'' of~$\ell$ as an Atkin prime is~$\log(\ell)/\log(n)$, and
  one should only keep the highest-valued primes for the final
  sieve. Thus, once a few Atkin primes have been collected, it only
  makes sense to look for low-degree factors of modular equations;
  this is cheaper than computing the full factorization.
  
\item \emph{Schoof's method.} When~$\ell = O(\sqrt{\log q})$ is not
  Elkies, it is usually more interesting to apply Schoof's original
  method to compute~$\chi(A)$ mod~$\ell$ than attempting to
  keep~$\ell$ as part of the Atkin data; this makes space for larger
  primes in the final sieve. If~$\ell$ is very small (for
  instance~$\ell=2$), then Schoof's method can also yield~$\chi(A)$
  modulo a suitable power
  of~$\ell$~\cite[§4]{gaudry_GenusPointCounting2012}.
\end{enumerate}

\bibliographystyle{abbrv}
\bibliography{counting-g2}

\begin{thebibliography}{10}

\bibitem{abelard_ImprovedComplexityBounds2019}
S.~Abelard, P.~Gaudry, and P.-J. Spaenlehauer.
\newblock Improved complexity bounds for counting points on hyperelliptic
  curves.
\newblock {\em Found. Comput. Math.}, 19(3):591--621, 2019.

\bibitem{ballentine_IsogeniesPointCounting2016}
S.~Ballentine, A.~Guillevic, E.~Lorenzo~Garc{\'i}a, C.~Martindale,
  M.~Massierer, B.~Smith, and J.~Top.
\newblock Isogenies for point counting on genus two hyperelliptic curves with
  maximal real multiplication.
\newblock In {\em Algebraic {{Geometry}} for {{Coding Theory}} and
  {{Cryptography}}}, volume~9, pages 63--94, {Los Angeles}, 2016. {Springer}.

\bibitem{bosma_MagmaAlgebraSystem1997}
W.~Bosma, J.~Cannon, and C.~Playoust.
\newblock The {{Magma}} algebra system. {{I}}. {{The}} user language.
\newblock {\em J Symb. Comput}, 24:235--265, 1997.

\bibitem{bostan_FastAlgorithmsComputing2008}
A.~Bostan, F.~Morain, B.~Salvy, and {\'E}.~Schost.
\newblock Fast algorithms for computing isogenies between elliptic curves.
\newblock {\em Math. Comp.}, 77(263):1755--1778, 2008.

\bibitem{broker_ModularPolynomialsGenus2009}
R.~Br{\"o}ker and K.~Lauter.
\newblock Modular polynomials for genus 2.
\newblock {\em LMS J. Comp. Math.}, 12:326--339, 2009.

\bibitem{broker_ModularPolynomialsIsogeny2012}
R.~Br{\"o}ker, K.~Lauter, and A.~V. Sutherland.
\newblock Modular polynomials via isogeny volcanoes.
\newblock {\em Math. Comp.}, 81:1201--1231, 2012.

\bibitem{broker_ExplicitHeightBound2010}
R.~Br{\"o}ker and A.~V. Sutherland.
\newblock An explicit height bound for the classical modular polynomial.
\newblock {\em Ramanujan J.}, 22(3):293--313, 2010.

\bibitem{brooks_IsogenyGraphsOrdinary2017}
E.~H. Brooks, D.~Jetchev, and B.~Wesolowski.
\newblock Isogeny graphs of ordinary abelian varieties.
\newblock {\em Res. Number Theory}, 3:28, 2017.

\bibitem{cohen_CoefficientsTransformationPolynomials1984}
P.~Cohen.
\newblock On the coefficients of the transformation polynomials for the
  elliptic modular function.
\newblock {\em Math. Proc. Cambridge Philos. Soc.}, 95(3):389--402, 1984.

\bibitem{dobson_TrustlessUnknownorderGroups2021}
S.~Dobson, S.~Galbraith, and B.~Smith.
\newblock Trustless unknown-order groups.
\newblock {\em Math. Crypt.}, 2021.

\bibitem{dudeanu_CyclicIsogeniesAbelian2017}
A.~Dudeanu, D.~Jetchev, D.~Robert, and M.~Vuille.
\newblock Cyclic isogenies for abelian varieties with real multiplication.
\newblock 2017.

\bibitem{elkies_EllipticModularCurves1998}
N.~D. Elkies.
\newblock Elliptic and modular curves over finite fields and related
  computational issues.
\newblock In {\em Computational perspectives on number theory ({{Chicago}},
  1995)}, volume~7, pages 21--76. {Amer. Math. Soc.}, 1998.

\bibitem{elkies_K3SurfacesEquations2014}
N.~D. Elkies and A.~Kumar.
\newblock K3 surfaces and equations for {{Hilbert}} modular surfaces.
\newblock {\em Algebra \& Number Theory}, 8(10):2297--2411, 2014.

\bibitem{enge_ComputingModularPolynomials2009}
A.~Enge.
\newblock Computing modular polynomials in quasi-linear time.
\newblock {\em Math. Comp.}, 78(267):1809--1824, 2009.

\bibitem{enge_ClassInvariantsCRT2010}
A.~Enge and A.~V. Sutherland.
\newblock Class invariants by the {CRT} method.
\newblock In {\em Proceedings of the 9th Algorithmic Number Theory Symposium
  (ANTS IX)}, pages 142--156, Nancy, 2010. Springer.

\bibitem{faltings_DegenerationAbelianVarieties1990}
G.~Faltings and C.-L. Chai.
\newblock {\em Degeneration of abelian varieties}.
\newblock {Springer}, 1990.

\bibitem{fouquet_IsogenyVolcanoesSEA2002}
M.~Fouquet and F.~Morain.
\newblock Isogeny volcanoes and the {{SEA}} algorithm.
\newblock In {\em Proceedings of the 5th {{Algorithmic Number Theory
  Symposium}} ({{ANTS V}})}, pages 276--291, {Sydney}, 2002. {Springer}.

\bibitem{gaudry_CountingPointsHyperelliptic2000}
P.~Gaudry and R.~Harley.
\newblock Counting points on hyperelliptic curves over finite fields.
\newblock In {\em Proceedings of the 4th {{Algorithmic Number Theory
  Symposium}} ({{ANTS IV}})}, pages 313--332, {Leiden}, 2000. {Springer}.

\bibitem{gaudry_CountingPointsGenus2011}
P.~Gaudry, D.~Kohel, and B.~Smith.
\newblock Counting points on genus 2 curves with real multiplication.
\newblock In {\em Advances in {{Cryptology}} -- {{Asiacrypt}} 2011}, pages
  504--519, {Seoul}, 2011. {Springer}.

\bibitem{gaudry_GenusPointCounting2012}
P.~Gaudry and {\'E}.~Schost.
\newblock Genus 2 point counting over prime fields.
\newblock {\em J. Symb. Comput.}, 47(4):368--400, 2012.

\bibitem{gundlach_BestimmungFunktionenZur1963}
K.-B. Gundlach.
\newblock Die {B}estimmung der {F}unktionen zur {H}ilbertschen {M}odulgruppe
  des {Z}ahlkörpers $\mathbb{Q}(\sqrt{5})$.
\newblock {\em Math. Ann.}, 152:226–256, 1963.

\bibitem{hart_FLINTFastLibrary}
W.~Hart.
\newblock {FLINT}: {F}ast {L}ibrary for {N}umber {T}heory.
\newblock \url{https://flintlib.org}.

\bibitem{harvey_CountingPointsHyperelliptic2014}
D.~Harvey.
\newblock Counting points on hyperelliptic curves in average polynomial time.
\newblock {\em Ann. of Math. (2)}, 179(2):783--803, 2014.

\bibitem{harvey_ComputingHasseWittMatrices2016}
D.~Harvey and A.~V. Sutherland.
\newblock Computing {{Hasse-Witt}} matrices of hyperelliptic curves in average
  polynomial time, {{II}}.
\newblock {\em Contemp. Math.}, 663:127--147, 2016.

\bibitem{igusa_ArithmeticVarietyModuli1960}
J.-I. Igusa.
\newblock Arithmetic variety of moduli for genus two.
\newblock {\em Ann. of Math. (2)}, 72:612--649, 1960.

\bibitem{ionica_IsogenyGraphsMaximal2020}
S.~Ionica and E.~Thom{\'e}.
\newblock Isogeny graphs with maximal real multiplication.
\newblock {\em J. Number Theory}, 207:385--422, 2020.

\bibitem{johansson_ArbEfficientArbitraryprecision2017}
F.~Johansson.
\newblock Arb: efficient arbitrary-precision midpoint-radius interval
  arithmetic.
\newblock {\em IEEE Trans. Comput.}, 66(8):1281--1292, 2017.

\bibitem{kedlaya_ComputingZetaFunctions2004}
K.~S. Kedlaya.
\newblock Computing zeta functions via $p$-adic cohomology.
\newblock In {\em Algorithmic Number Theory}, pages 1--17. Springer, 2004.

\bibitem{kieffer_HDMELibraryEvaluation}
J.~Kieffer.
\newblock {HDME}: a {C} library for the evaluation of modular equations in
  dimension 2.
\newblock \url{https://github.com/j-kieffer/hdme}.

\bibitem{kieffer_EvaluatingModularEquations2021}
J.~Kieffer.
\newblock Evaluating modular equations for abelian surfaces.
\newblock 2021.

\bibitem{kieffer_DegreeHeightEstimates2022}
J.~Kieffer.
\newblock Degree and height estimates for modular equations on {PEL} {S}himura
  varieties.
\newblock {\em J. London Math. Soc.}, 2022.

\bibitem{kieffer_ComputingIsogeniesModular2019}
J.~Kieffer, A.~Page, and D.~Robert.
\newblock Computing isogenies from modular equations in genus two.
\newblock 2019.

\bibitem{koblitz_EllipticCurveCryptosystems1987}
N.~Koblitz.
\newblock Elliptic curve cryptosystems.
\newblock {\em Math. Comp.}, 48(177):203--209, 1987.

\bibitem{koblitz_HyperellipticCryptosystems1989}
N.~Koblitz.
\newblock Hyperelliptic cryptosystems.
\newblock {\em J. Cryptol.}, 1:139--150, 1989.

\bibitem{martindale_HilbertModularPolynomials2020}
C.~Martindale.
\newblock Hilbert modular polynomials.
\newblock {\em J. Number Theory}, 213:464--498, 2020.

\bibitem{milio_DatabaseModularPolynomials}
E.~Milio.
\newblock Database of modular polynomials of {H}ilbert and {S}iegel.
\newblock \url{https://members.loria.fr/EMilio/modular-polynomials}.

\bibitem{milio_QuasilinearTimeAlgorithm2015}
E.~Milio.
\newblock A quasi-linear time algorithm for computing modular polynomials in
  dimension 2.
\newblock {\em LMS J. Comput. Math.}, 18:603--632, 2015.

\bibitem{milio_ModularPolynomialsHilbert2020}
E.~Milio and D.~Robert.
\newblock Modular polynomials on {H}ilbert surfaces.
\newblock {\em J. Number Theory}, 216:403--459, 2020.

\bibitem{milne_AbelianVarieties1986}
J.~S. Milne.
\newblock Abelian varieties.
\newblock In {\em Arithmetic geometry ({{Storrs}}, 1984)}, pages 103--150.
  {Springer}, 1986.

\bibitem{pila_FrobeniusMapsAbelian1990}
J.~Pila.
\newblock Frobenius maps of abelian varieties and finding roots of unity in
  finite fields.
\newblock {\em Math. Comp.}, 55(192):745--763, 1990.

\bibitem{ruck_AbelianSurfacesJacobian1990}
H.-G. R{\"u}ck.
\newblock Abelian surfaces and {{Jacobian}} varieties over finite fields.
\newblock {\em Compos. Math.}, 76(3):351--366, 1990.

\bibitem{schoof_EllipticCurvesFinite1985}
R.~Schoof.
\newblock Elliptic curves over finite fields and the computation of square
  roots mod $p$.
\newblock {\em Math. Comp.}, 44(170):483–494, 1985.

\bibitem{schoof_CountingPointsElliptic1995}
R.~Schoof.
\newblock Counting points on elliptic curves over finite fields.
\newblock {\em J. Th\'eorie Nr. Bordx.}, 7(1):219--254, 1995.

\bibitem{shparlinski_ProductSmallElkies2015}
I.~E. Shparlinski.
\newblock On the product of small {{Elkies}} primes.
\newblock {\em Proc. Amer. Math. Soc.}, 143(4):1441--1448, 2015.

\bibitem{shparlinski_DistributionAtkinElkies2014}
I.~E. Shparlinski and A.~V. Sutherland.
\newblock On the distribution of {{Atkin}} and {{Elkies}} primes.
\newblock {\em Found. Comput. Math.}, 14:285--297, 2014.

\bibitem{shparlinski_DistributionAtkinElkies2015}
I.~E. Shparlinski and A.~V. Sutherland.
\newblock On the distribution of {{Atkin}} and {{Elkies}} primes for reductions
  of elliptic curves on average.
\newblock {\em LMS J. Comput. Math.}, 18:308--322, 2015.

\bibitem{streng_ComputingIgusaClass2014}
M.~Streng.
\newblock Computing {{Igusa}} class polynomials.
\newblock {\em Math. Comp.}, 83:275--309, 2014.

\bibitem{sutherland_DatabaseModularPolynomials}
A.~V. Sutherland.
\newblock Database of modular polynomials of prime levels up to~$1000$ for
  the~$j$-function and up to~$5000$ for the weber~$f$ function.
\newblock \url{https://math.mit.edu/~drew/}.

\bibitem{sutherland_EvaluationModularPolynomials2013}
A.~V. Sutherland.
\newblock On the evaluation of modular polynomials.
\newblock In {\em Proceedings of the 10th {{Algorithmic Number Theory
  Symposium}}}, pages 531--555, {San Diego}, 2013. {Math. Sci. Publ.}

\bibitem{sutherland_CountingPointsSuperelliptic2020}
A.~V. Sutherland.
\newblock Counting points on superelliptic curves in average polynomial time.
\newblock In {\em Proceedings of the 14th {{Algorithmic Number Theory
  Symposium}}}, pages 403--422. {Math. Sci. Publ.}, 2020.

\bibitem{theparigroup_PariGPVersion2019}
{The PARI group}.
\newblock Pari/{GP} version 2.11.0.
\newblock \url{http://pari.math.u-bordeaux.fr/}, 2019.

\bibitem{villard_ComputingResultantGeneric2018a}
G.~Villard.
\newblock On computing the resultant of generic bivariate polynomials.
\newblock In {\em Proceedings of the 2018 {{ACM International Symposium}} on
  {{Symbolic}} and {{Algebraic Computation}}}, pages 391--398, {New York},
  2018. {ACM}.

\bibitem{weil_CourbesAlgebriquesVarietes1948}
A.~Weil.
\newblock {\em {Sur les courbes alg\'ebriques et les vari\'et\'es qui s'en
  d\'eduisent}}.
\newblock {Hermann}, 1948.

\bibitem{weng_ConstructingHyperellipticCurves2003}
A.~Weng.
\newblock Constructing hyperelliptic curves of genus 2 suitable for
  cryptography.
\newblock {\em Math. Comput.}, 72(241):435--458, 2003.

\end{thebibliography}

\end{document}